\documentclass[10pt,reqno]{amsart}
\usepackage{amssymb,amscd}
\usepackage{color}   

\newtheorem{theorem}{Theorem}
\newtheorem{proposition}[theorem]{Proposition} \newtheorem{corollary}[theorem]{Corollary}
\newtheorem{lemma}[theorem]{Lemma}

\theoremstyle{definition}

\newtheorem{remark}[theorem]{Remark}

\DeclareMathOperator{\Ad}{Ad}
\DeclareMathOperator{\Aut}{Aut}
\DeclareMathOperator{\Int}{Int}
\DeclareMathOperator{\Id}{Id}
\DeclareMathOperator{\Ker}{Ker}
\DeclareMathOperator{\Sp}{Sp}

\def\botimes{\mathbin{\bar{\otimes}}}

\begin{document}
\title[Continuous cores of type III$_1$ free product factors]{
A characterization of fullness of continuous cores of type III$_1$ free product factors 
}
\author[R.~Tomatsu]{Reiji Tomatsu$\,^{1}$}
\address{
(R.T.) Department of Mathematics, Hokkaido University, Hokkaido, 060-0810, Japan
}
\email{tomatsu@math.sci.hokudai.ac.jp}
\thanks{$\,^{1}$Supported by Grant-in-Aid for Young Scientists (B) 24740095.}
\author[Y.~Ueda]{Yoshimichi Ueda$\,^{2}$}
\address{
(Y.U.) Graduate School of Mathematics, 
Kyushu University, 
Fukuoka, 810-8560, Japan
}
\email{ueda@math.kyushu-u.ac.jp}
\thanks{$\,^{2}$Supported by Grant-in-Aid for Scientific Research (C) 24540214.}
\date{Dec.~10, 2014}

\begin{abstract} 
We prove that, for any type III$_1$ free product factor, its continuous core is full if and only if its $\tau$-invariant is the usual topology on the real line. This trivially implies, as a particular case, the same result for free Araki--Woods factors. Moreover, our method shows the same result for full (generalized) Bernoulli crossed product factors of type III$_1$.
\end{abstract}

\maketitle

\allowdisplaybreaks{

\section{Introduction} 

Let $M_1, M_2$ be two non-trivial von Neumann algebras with separable preduals and $\varphi_1, \varphi_2$ be faithful normal states on them, respectively. Let $(M,\varphi)=(M_1,\varphi_1)\star(M_2,\varphi_2)$ be their free product (see e.g.~\cite[\S\S2,1]{Ueda:AdvMath11}). Then $M$ must be of the form $M = M_d\oplus M_c$ or $M_c$, where $M_d$ is finite dimensional (which can explicitly be determined) and $M_c$ is diffuse. In what follows, we assume that $(\dim M_1, \dim M_2) \neq (2,2)$; otherwise $M_c = L^\infty[0,1]\botimes M_2(\mathbb{C})$. Then $M_c$ is a full factor of type II$_1$ (if both $\varphi_i$ are tracial), III$_\lambda$ with $0<\lambda<1$ (if the modular actions $\sigma^{\varphi_i}$ have a common (positive) period and the smallest one is $2\pi/|\log\lambda|$), or III$_1$ (otherwise); hence we call $M_c$ a {\it free product factor} in what follows. Moreover, Connes's $\tau$-invariant $\tau(M_c)$ (see \cite{Connes:JFA74}) coincides with the weakest topology on $\mathbb{R}$ making the mapping $t \in \mathbb{R} \mapsto \sigma_t^\varphi \in \Aut(M)$ (equipped with the so-called $u$-topology, see e.g.~\cite[\S III]{Connes:JFA74}) continuous. See \cite[Theorem 4.1]{Ueda:AdvMath11},\cite[Theorem 3.1]{Ueda:MRL11} for these facts. In this way, almost all the basic invariants have been made clear for $M$ (and $M_c$), but it still remains an open question when the continuous core of $M_c$ becomes a full factor (if $M_c$ is of type III$_1$). Here, for a given type III von Neumann algebra we call the `carrier algebra' of its so-called associated covariant system (see \cite[Definition XII.1.3 and XII.1.5]{Takesaki:Book2}) its {\it continuous core}. In this note, we would like to report the following simple solution to the question:  

\begin{theorem}\label{T1} Assume that $M_c$ is of type III$_1$. Then the following conditions are equivalent{\rm:}
\begin{itemize} 
\item[(1)] The continuous core $\widetilde{M_c} := M_c\rtimes_{\sigma^{\varphi_c}}\mathbb{R}$ with $\varphi_c := \varphi\!\upharpoonright_{M_c}$ is full. 
\item[(2)] The $\tau$-invariant $\tau(M_c)$, i.e., the weakest topology on $\mathbb{R}$ making the mapping $t \in \mathbb{R} \mapsto \sigma_t^\varphi \in
\Aut(M)$ continuous in this particular case, is the usual topology on $\mathbb{R}$. 
\item[(3)] For any sequence $t_n$ in $\mathbb{R}$ we have{\rm:} $(\sigma_{t_n}^{\varphi_1},\sigma_{t_n}^{\varphi_2}) \longrightarrow (\Id_{M_1},\Id_{M_2})$ in $\Aut(M_1)\times\Aut(M_2)$ as $n \to\infty$ implies $t_n \longrightarrow 0$ in the usual topology on $\mathbb{R}$ as $n\to\infty$.
\end{itemize} 
\end{theorem} 

The above theorem completes the project to compute all the basic invariants for arbitrary free product von Neumann algebras. (Here we would like to mention that the triviality of the asymptotic bicentralizer of any type III$_1$ free product factor was confirmed by the second-named author using only \cite[Theorem 4, Corollary 8]{ConnesStormer:JFA78} and \cite[Corollary 3.2, Theorem 4.1]{Ueda:AdvMath11}.) One of the important features of Theorem \ref{T1} is that the consequence is formulated in terms of modular automorphisms associated with given states rather than the $\tau$-invariant itself; hence it is suitable for practical use. Moreover, the next corollary is obtained as a particular case of the theorem; see Remark \ref{R9}.

\begin{corollary}\label{C2} Let $\Gamma(\mathcal{H}_\mathbb{R},U_t)''$ be a free Araki--Woods factor of type III$_1$ {\rm(}\cite{Shlyakhtenko:PacificJMath97}{\rm)}. Then the continuous core of $\Gamma(\mathcal{H}_\mathbb{R},U_t)''$ is full if and only if the weakest topology on $\mathbb{R}$ making $t \mapsto U_t$ {\rm(}with respect to the strong operator topology{\rm)} be the usual one. 
\end{corollary} 

There are previously known cases where the continuous cores of free Araki--Woods factors become full; see Shlyakhtenko \cite[Theorem 4.8]{Shlyakhtenko:Crelle98}, Houdayer \cite[Theorem 1.2]{Houdayer:JIMJ10}, and more recent Houdayer--Raum \cite[Theorem B]{HoudayerRaum:arXiv:1406.6160} ({\it n.b.}~the second needs \cite[Proposition 7]{Ozawa:ActaMath04} with $\mathcal{N}_0 = \mathcal{M}$ there). However, any kind of characterization such as the above corollary has never been known so far. 

\medskip
Another important class of full factors of type III$_1$ whose $\tau$-invariants are already computed consists of so-called Bernoulli crossed products. In fact, Vaes and Verraedt \cite[\S\S2.5]{VaesVerraedt:arXiv:1408.6414} recently proved that any Bernoulli crossed product of non-amenable group must be a full factor, and moreover computed its $\tau$-invariant in terms of given data, generalizing Connes's original work \cite{Connes:JFA74}. In the appendix, we will explain that our method of proving Theorem \ref{T1} works well even for (generalized) Bernoulli crossed products; see Theorem \ref{T11} for the precise assertion. 

\medskip   
This note uses the same standard notation as in \cite{Ueda:AdvMath11},\cite{Ueda:MRL11} (except the appendix, where the notation follows \cite[\S\S2.5]{VaesVerraedt:arXiv:1408.6414}). We will freely use (Ocneanu) ultraproducts and asymptotic centralizers (denoted by $N^\omega \supseteq N_\omega$, respectively, for given von Neumann algebras $N$), for which we refer to \cite[\S\S2.2]{Ueda:AdvMath11} as a brief summary and to \cite{AndoHaagerup:JFA14} as a detailed reference. Our discussion below is fairly simple, though it depends upon some previous works \cite{Ueda:JLMS13}, \cite[\S\S2.1]{Ueda:arXiv:1207.6838v3} (based on \cite[Theorem 4.1]{Ueda:AdvMath11}), \cite[\S\S2.5]{VaesVerraedt:arXiv:1408.6414} and the automorphism analysis due to Connes and Ocneanu.
 
\section{Preliminary Facts}

Let us start with a general lemma on group actions on factors. Our intuition about it came from quite a recent result \cite[Theorem 7.7]{MasudaTomatsu:arXiv:1206.0955v2}. 


\begin{lemma}\label{L3} Let $\alpha \colon \Gamma \curvearrowright N$ be an action of a countable discrete abelian group on a factor with separable predual. Let $\alpha_\omega \colon \Gamma \curvearrowright N_\omega$ be the action on the asymptotic centralizer $N_\omega$ arising from $\alpha$. Then, for every $p \in \Ker(\alpha_\omega)^\perp$ {\rm(}in the dual $\widehat{\Gamma}${\rm)} there exists a unitary $u\in N_\omega$ such that $\alpha_{\omega,\gamma}(u)=\langle \gamma,p\rangle u$ holds for all $\gamma\in \Gamma$, where $\langle\,\cdot\,,\,\cdot\,\rangle$ is the dual pairing between $\Gamma$ and $\widehat{\Gamma}$ and $\Lambda^\perp := \{ p \in \widehat{\Gamma} \mid \langle \Lambda, p \rangle = 0\}$ for a subgroup $\Lambda$ of $\Gamma$. Moreover, for every $p \in \Ker(\alpha_\omega)^\perp$, the dual action $\widehat{\alpha}_p$ is approximately inner, that is, it falls in the closure of $\Int(N\rtimes_\alpha\Gamma)$. 
\end{lemma}
\begin{proof}
By \cite[Proposition 2.1.2]{Connes:ASENS75}, the action $\alpha_\omega$ induces a properly outer action of $\Gamma/\Ker(\alpha_\omega)$ on $N_\omega$. Note that the dual of $\Gamma/\Ker(\alpha_\omega)$ is naturally identified with $\Ker(\alpha_\omega)^\perp$ in $\widehat{\Gamma}$. Let $p \in \Ker(\alpha_\omega)^\perp$ be arbitrarily chosen. We apply the so-called $1$-cohomology vanishing theorem \cite[\S7.2]{Ocneanu:LMN1138} to the (rather simple) cocycle $\gamma \mapsto \langle \gamma,p\rangle1 \in N_\omega$ with the above properly outer action, and get the desired unitary $u \in N_\omega$. 

Since $u \in N_\omega$ ($\subseteq N'\cap N^\omega$ trivially), one easily observes that $\widehat{\alpha}_p(x) = u x u^*$ holds inside $(N\rtimes_\alpha\Gamma)^\omega$ for every $x \in N\rtimes_\alpha\Gamma$. Thanks to \cite[Proposition 1.1.3 (b)]{Connes:ASENS75} we can choose a representing sequence $u_n$ of $u$ in such a way that it consists of unitaries. Let $\psi$ be a faithful normal state on $N$, and set $\widetilde{\psi}:=\psi\circ E_N$
with the canonical conditional expectation
$E_N \colon N\rtimes_\alpha\Gamma \to N$. For any $y, z \in N\rtimes_\alpha\Gamma$ one has 
\begin{align*}
|((y\widetilde{\psi})\circ\widehat{\alpha}_p - (y\widetilde{\psi})\circ
\Ad u_n)(z)| 
&\leq
\Vert\widehat{\alpha}_p^{-1}(y) - u_n^* y u_n\Vert_{\widetilde{\psi}}\,\Vert z\Vert_\infty + |\widetilde{\psi}(u_n z u_n^* y) - \widetilde{\psi}(zu_n^* y u_n)| \\
&\leq 
\Vert\widehat{\alpha}_p^{-1}(y)- u_n^* y u_n\Vert_{\widetilde{\psi}}\,\Vert z\Vert_\infty
+
\Vert \psi u_n -u_n \psi\Vert\,\Vert y\Vert_\infty \Vert z\Vert_\infty
\end{align*} 
so that 
\[
\Vert (y\widetilde{\psi})\circ\widehat{\alpha}_p -
(y\widetilde{\psi})\circ\Ad u_n\Vert
\leq 
\Vert\widehat{\alpha}_p^{-1}(y)-u_n^* y u_n\Vert_{\widetilde{\psi}} + \Vert \psi u_n -u_n \psi\Vert\,\Vert y\Vert_\infty \longrightarrow 0
\]
as $n \to \omega$. Therefore, $\widehat{\alpha}_p = \lim_{n\to\omega}\Ad u_n$ in $\Aut(N\rtimes_\alpha\Gamma)$, because the $y\widetilde{\psi}$ form a dense subset of the predual.  
\end{proof}

For a given type III$_\lambda$ factor, we call the canonical type II$_\infty$ factor $\mathcal{N}_0$ in \cite[Theorem XII.2.1]{Takesaki:Book2} the {\it discrete core} of the type III$_\lambda$ factor. The discrete core is indeed uniquely determined from the given type III$_\lambda$ factor. Moreover, it can explicitly be constructed based on the Takesaki duality; see the proof of \cite[Theorem XII.2.1]{Takesaki:Book2}). 

\begin{proposition}\label{P4} Let $\lambda \in (0,1)$ and set $T := 2\pi/|\log\lambda|$. Let $Q$ be a type III$_\lambda$ factor with separable predual. Then $Q$ is full if and only if so is its discrete core $\widehat{Q} := Q\rtimes_{\sigma^\chi}(\mathbb{R}/T\mathbb{Z})$ with a periodic state $\chi$ {\rm(}i.e., a faithful normal state with $\sigma_T^\chi = \Id_Q${\rm)}. 
\end{proposition} 
\begin{proof} (The `if' part.) Assume that $\widehat{Q}$ is full. It is known that $\widehat{Q}$ is stably isomorphic to $Q_\chi$. Hence $Q_\chi$ is also full. By \cite[Proposition 2.3 (2)]{Connes:JFA74} $Q$ must be full. 

\medskip
(The `only if' part.) Assume next that $Q$ is full. Let us denote by $\theta \colon \mathbb{Z} \curvearrowright \widehat{Q}$ the dual action of $\sigma^\chi\colon\mathbb{R}/T\mathbb{Z} \curvearrowright Q$. 

Suppose that $\theta_\omega$ is a non-trivial action. By Lemma \ref{L3} there exists $\zeta \in \mathbb{T}\setminus\{1\}$ so that the dual action $\widehat{\theta}_\zeta$ falls in the closure of $\Int(\widehat{Q}\rtimes_\theta\mathbb{Z})$. Since $\widehat{Q}\rtimes_\theta\mathbb{Z} \cong Q$ is full, we conclude that $\widehat{\theta}_\zeta$ must be inner. However, $\widehat{\theta}$ is the bi-dual action of $\sigma^\chi \colon [0,T) = \mathbb{R}/T\mathbb{Z} = \mathbb{T} \curvearrowright Q$, and therefore, by \cite[Theorem X.2.3 (iv)]{Takesaki:Book2} $\sigma^\chi_t$ is inner for some $0 < t < T$, a contradiction. Hence we have shown that $\theta_\omega$ is indeed the trivial action. 

Let $v \in \widehat{Q}_\omega$ be an arbitrary unitary. Since $\theta_\omega$ is trivial, we observe that $x=vxv^*$ inside $(\widehat{Q}\rtimes_\theta\mathbb{Z})^\omega$ for every $x \in \widehat{Q}\rtimes_\theta\mathbb{Z}$.
The same argument as that for getting
$\widehat{\alpha}_p = \lim_{n\to\omega}\Ad u_n$ in the proof of Lemma \ref{L3} shows that $v \in (\widehat{Q}\rtimes_\theta\mathbb{Z})_\omega \cong Q_\omega = \mathbb{C}1$. 
\end{proof} 

The proof of Proposition \ref{P4} (especially, its `only if' part) actually works well, without any essential change, for showing the next proposition. Note that the discrete decomposition is well-defined for any full type III$_1$ factor as long as it is possible; see \cite{Connes:JFA74} (and also \cite[\S\S2.2]{Ueda:arXiv:1207.6838v3} for its explicit construction based on the Takesaki duality). 

\begin{proposition}\label{P5} 
The discrete core of any full type III$_1$ factor  must be full {\rm(}if it exists{\rm)}. 
\end{proposition}

Our question is about the fullness of certain {\it continuous} crossed product factors {\it of type II$_\infty$}, but the next lemma says that it is equivalent to that of certain {\it discrete} crossed product factors {\it of type III$_\lambda$}. 

\begin{lemma}\label{L6} Let $P$ be a type III$_1$ factor with separable predual and  $\chi$ be a faithful normal state on it. Let $\lambda \in (0,1)$ be arbitrarily chosen and set $T := 2\pi/|\log\lambda|$. Then the continuous core $\widetilde{P}$ is full if and only if so is the type III$_\lambda$ factor $Q := P\rtimes_{\sigma_T^\chi}\mathbb{Z}$. 
\end{lemma} 
\begin{proof} By \cite[Lemma XVIII.4.17 (i)]{Takesaki:Book3} and its proof we have known that $Q$ is indeed a type III$_\lambda$ factor and that its discrete core $\widehat{Q}$ is identified with the continuous core $\widetilde{P}$. Thus the desired assertion immediately follows from Proposition \ref{P4}. 
\end{proof} 

\section{Proof of Theorem \ref{T1}}

Our main concern is to prove (2) $\Rightarrow$ (1) of Theorem \ref{T1}. If both $M_1, M_2$ are (possibly infinite) direct sums of type I factors, then both $\varphi_1, \varphi_2$ are almost periodic and so is the positive linear functional $\varphi_c$ (see \cite[Theorem 2.1]{Ueda:MRL11}); hence $\tau(M_c)$ does never become the usual topology on $\mathbb{R}$. Therefore, we may and do assume that $M_1$ has a diffuse direct summand. Note here that $\tau$-invariant is a von Neumann algebraic invariant. Hence, by the trick explained at the beginning of \cite[\S\S2.1]{Ueda:arXiv:1207.6838v3} we may and do further assume that $M_1$ is either (a) a diffuse von Neumann algebra with no type III$_1$ factor direct summands or (b) a type III$_1$ factor. In each case, $M=M_c$ holds thanks to \cite[Theorem 4.1]{Ueda:AdvMath11}. In what follows, we fix $\lambda \in (0,1)$ and set $T := 2\pi/|\log\lambda|$, and it suffices, thanks to Lemma \ref{L6}, to prove that $M\rtimes_{\sigma_T^\varphi}\mathbb{Z}$ is full under condition (2) of Theorem \ref{T1}. We need two technical lemmas. 

\begin{lemma}\label{L7} With the conditional expectation
$E_{\varphi_1} := (\varphi_1\botimes \Id)\!\upharpoonright_{M_1\rtimes_{\sigma_T^{\varphi_1}}\mathbb{Z}}\colon M_1\rtimes_{\sigma_T^{\varphi_1}}\mathbb{Z} \to \mathbb{C}1\rtimes_{\sigma_T^{\varphi_1}}\mathbb{Z}$, one can find a faithful normal state $\psi$ on $\mathbb{C}1\rtimes_{\sigma_T^{\varphi_1}}\mathbb{Z}$ so that for each natural number $n \geq 2$ there exists a unitary $u_n \in (M_1\rtimes_{\sigma_T^{\varphi_1}}\mathbb{Z})_{\psi\circ E_{\varphi_1}}$ such that $E_{\varphi_1}(u_n^k) = 0$ as long as $1 \leq k \leq n-1$. 
\end{lemma}   
\begin{proof} 
We first treat case (a). It is easy to see that $(M_1)_{\varphi_1}$ is diffuse; see the proof of \cite[Theorem 3.4]{Ueda:AdvMath11} with standard facts (see e.g.~\cite[Lemma 11, Lemma 12]{Ueda:MathScand01}). Let $E_{M_1}\colon M_1\rtimes_{\sigma_T^{\varphi_1}}\mathbb{Z} \to M_1$ be the canonical conditional expectation. One can easily confirm $\tau_\mathbb{Z}\circ E_{\varphi_1} = \varphi_1\circ E_{M_1}$ with the canonical tracial state $\tau_\mathbb{Z}$ on $L(\mathbb{Z}) = \mathbb{C}1\rtimes_{\sigma_T^{\varphi_1}}\mathbb{Z}$ naturally, and thus $(M_1)_{\varphi_1}\botimes L(\mathbb{Z})$ naturally sits in $(M_1\rtimes_{\sigma_T^{\varphi_1}}\mathbb{Z})_{\tau_\mathbb{Z}\circ E_{\varphi_1}}$. Let $v \in (M_1)_{\varphi_1}$ be a Haar unitary with respect to $\varphi_1$ (see e.g.~the proof of \cite[Theorem 3.7]{Ueda:AdvMath11} for its existence), and $\psi := \tau_\mathbb{Z}$ and $u_n := v\otimes1$ (for every $n$) are the desired ones. 

We then treat case (b). Let us denote by $u \in \mathbb{C}1\rtimes_{\sigma_T^{\varphi_1}}\mathbb{Z}$ the canonical unitary generator. By a standard fact (see e.g.~\cite[Theorem X.1.17]{Takesaki:Book2}) together with the identity $\tau_\mathbb{Z}\circ E_{\varphi_1} = \varphi_1\circ E_{M_1}$ we observe that $\sigma_T^{\tau_\mathbb{Z}\circ E_\varphi} = \Ad u$. One can choose a positive invertible $h \in \mathbb{C}1\rtimes_{\sigma_T^{\varphi_1}}\mathbb{Z}$ so that $u^* = h^{iT}$. Set $\psi := \tau_{\mathbb{Z}}(h)^{-1}\tau_\mathbb{Z}(h\,-\,)$, and one has $\sigma_T^{\psi\circ E_{\varphi_1}} = \Id_{M_1\rtimes_{\sigma_T^{\varphi_1}}\mathbb{Z}}$.
By \cite[Theorem 4.2.6]{Connes:ASENS73} $(M_1\rtimes_{\sigma_T^{\varphi_1}}\mathbb{Z})_{\psi\circ E_\varphi}$ must be a type II$_1$ factor and contain $\mathbb{C}1\rtimes_{\sigma_T^{\varphi_1}}\mathbb{Z}$. Since $\mathbb{C}1\rtimes_{\sigma_T^{\varphi_1}}\mathbb{Z}$ is diffuse, for each natural number $n \geq 2$ there exist $n$ orthogonal $e_0,\dots,e_{n-1} \in (\mathbb{C}1\rtimes_{\sigma_T^{\varphi_1}}\mathbb{Z})^p$, all of which are equivalent in $(M_1\rtimes_{\sigma_T^{\varphi_1}}\mathbb{Z})_{\psi\circ E_\varphi}$, and $\sum_{i=0}^{n-1}e_i = 1$. Then one can construct a unitary $u_n \in (M_1\rtimes_{\sigma_T^{\varphi_1}}\mathbb{Z})_{\psi\circ E_\varphi}$ in such a way that $u_n e_0 = e_1 u_n$, $u_n e_1 = e_2 u_n$, $\dots, u_n e_{n-1} = e_0 u_n$. Since $\mathbb{C}1\rtimes_{\sigma_T^{\varphi_1}}\mathbb{Z}$ is commutative, one has $E_{\varphi_1}(u_n^k) = 0$ for every $1 \leq k \leq n-1$. 
\end{proof} 

\begin{lemma}\label{L8} We have $(M\rtimes_{\sigma_T^\varphi}\mathbb{Z})_\omega = (M\rtimes_{\sigma_T^\varphi}\mathbb{Z})' \cap (M\rtimes_{\sigma_T^\varphi}\mathbb{Z})^\omega = M' \cap (\mathbb{C}1\rtimes_{\sigma_T^\varphi}\mathbb{Z})^\omega$, where $M$ canonically sits in $M\rtimes_{\sigma_T^\varphi}\mathbb{Z}$. 
\end{lemma}
\begin{proof} Similarly to \cite[Theorem 5.1]{Ueda:PacificJMath99} we have 
\[
(M\rtimes_{\sigma_T^\varphi}\mathbb{Z},E_\varphi) = (M_1\rtimes_{\sigma_T^{\varphi_1}}\mathbb{Z},E_{\varphi_1})\star_{\mathbb{C}1\rtimes_{\sigma_T^\varphi}\mathbb{Z}}(M_2\rtimes_{\sigma_T^{\varphi_2}}\mathbb{Z},E_{\varphi_2})  
\]
naturally, to which \cite[Proposition 3.5]{Ueda:JLMS13} is applicable thanks to Lemma \ref{L7}. Since $M_2$ is non-trivial, one can choose an invertible $y \in \Ker(\varphi_2)$ so that $E_{\varphi_2}(y^* y) = \varphi_2(y^* y)1 \neq 0$. Therefore, \cite[Proposition 3.5]{Ueda:JLMS13} actually says that $(M\rtimes_{\sigma_T^\varphi}\mathbb{Z})_\omega \subseteq (M\rtimes_{\sigma_T^\varphi}\mathbb{Z})'\cap(M\rtimes_{\sigma_T^\varphi}\mathbb{Z})^\omega \subseteq (M_1\rtimes_{\sigma_T^{\varphi_1}}\mathbb{Z})^\omega$. For any $x \in (M_2\rtimes_{\sigma_T^{\varphi_2}}\mathbb{Z})' \cap (M_1\rtimes_{\sigma_T^{\varphi_1}}\mathbb{Z})^\omega$ one has $y(x-E_\varphi^\omega(x)) + yE_\varphi^\omega(x) = yx = xy = E_\varphi^\omega(x)y + (x-E_\varphi^\omega(x))y$, and the free independence of $(M_1\rtimes_{\sigma_T^{\varphi_1}}\mathbb{Z})^\omega$ and $(M_2\rtimes_{\sigma_T^{\varphi_2}}\mathbb{Z})^\omega$ in $((M\rtimes_{\sigma_T^\varphi}\mathbb{Z})^\omega, E_\varphi^\omega)$ (see \cite[Proposition 4]{Ueda:TAMS03}) forces at least $y(x-E_1^\omega(x)) = 0$; implying $x = E_\varphi^\omega(x) \in (\mathbb{C}1\rtimes_{\sigma_T^\varphi}\mathbb{Z})^\omega$ thanks to the invertibility of $y$. Consequently, $(M\rtimes_{\sigma_T^\varphi}\mathbb{Z})_\omega \subseteq (\mathbb{C}1\rtimes_{\sigma_T^\varphi}\mathbb{Z})^\omega$, from which the desired assertion immediately follows. 
\end{proof} 

We are ready to prove the desired assertion. 

\medskip\noindent
{\bf Proof of Theorem 1 (2) $\Rightarrow$ (1):} We prove its contraposition. Namely, assume that $\widetilde{M}$ is not full. Lemma \ref{L6} together with \cite[Theorem XIV.3.8, Theorem XIV.4.7]{Takesaki:Book2} says that $(M\rtimes_{\sigma_T^\varphi}\mathbb{Z})_\omega \neq \mathbb{C}1$. 

\medskip\noindent
{\bf Claim:} There exists a sequence $k_l$ in $\mathbb{Z}\setminus\{0\}$ such that $\sigma_{k_l T}^\varphi \longrightarrow \Id_M$ in $\Aut(M)$ as $l\to\infty$, or equivalently, that $\Vert x - \sigma_{k_l T}^\varphi(x)\Vert_\varphi \longrightarrow 0$ as $l\to\infty$ for every $x \in M$ (since $\sigma_t^\varphi$ preserves $\varphi$; see \cite[Theorem IX.1.15, Proposition IX.1.17]{Takesaki:Book2}). 

\medskip\noindent
(Proof of Claim) On the contrary, suppose that there exist $\varepsilon>0$ and a finite subset $\mathfrak{F}$ of $M$ such that $\sum_{y \in \mathfrak{F}} \Vert y - \sigma_{mT}^\varphi(y)\Vert_\varphi^2 \geq \varepsilon$ as long as $m\neq0$. Let $x \in (M\rtimes_{\sigma_T^\varphi}\mathbb{Z})_\omega$ be arbitrarily chosen with a representing sequence $x_n$. Lemma \ref{L8} shows that $x$ falls in $(\mathbb{C}1\rtimes_{\sigma_T^\varphi}\mathbb{Z})^\omega$ so that we can approximate each $x_n$ in the $\sigma$-strong topology by a bounded net consisting of finite linear combinations of the form $\sum_m c_m u^m$ with scalars $c_m$ (see the proof of Lemma \ref{L7} for the symbol `$u$'). We have
\begin{align*} 
\big\Vert \big(\sum_m c_m u^m\big) &- \tau_\mathbb{Z}\big(\sum_m c_m u^m\big)1\big\Vert_{\tau_\mathbb{Z}}^2 
= 
\sum_{m\neq0} |c_m|^2 \\
&\leq 
\varepsilon^{-1} \sum_{m\neq0} |c_m|^2 \sum_{y \in \mathfrak{F}} \Vert y - \sigma_{mT}^\varphi(y)\Vert_\varphi^2 \\
&\leq  
\varepsilon^{-1}\sum_{y\in\mathfrak{F}}\big\Vert\sum_m c_m(y - \sigma_{mT}^\varphi(y))u^m\big\Vert_{\varphi\circ E}^2 \\
&= 
\varepsilon^{-1} \sum_{y\in\mathfrak{F}} \big\Vert y\big(\sum_m c_m u^m\big) - \big(\sum_m c_m u^m\big)y\big\Vert_{\varphi\circ E}^2.   
\end{align*}
It follows that $\Vert x_n - \tau_\mathbb{Z}(x_n)1\Vert_{\tau_\mathbb{Z}}^2 \leq \varepsilon^{-1} \sum_{y\in\mathfrak{F}} \Vert yx_n -  x_n y\big\Vert_{\varphi\circ E}^2$ for every $n$. Taking the limit of this inequality as $n\to\omega$ we get 
\[
0 \leq \Vert x - \tau_{\mathbb{Z}}^\omega(x)1\Vert_{\tau_\mathbb{Z}^\omega} \leq \varepsilon^{-1} \sum_{y\in\mathfrak{F}} \Vert y x - x y \Vert_{(\varphi\circ E)^\omega}^2 = 0;
\]
implying $x = \tau_{\mathbb{Z}}^\omega(x)1$, a contradiction to $(M\rtimes_{\sigma_T^\varphi}\mathbb{Z})_\omega \neq \mathbb{C}1$. Hence we have proved the claim. 

\medskip
Since $|k_l T| \geq T > 0$ for all $l$, the sequence $k_l T$ in the claim does never converge to $0$ in the usual topology on $\mathbb{R}$. Nevertheless, $\sigma_{k_l T}^\varphi \longrightarrow \Id_M$ in $\Aut(M)$ as $l \to \infty$. These clearly contradict condition (2). Hence we are done. \qed

\bigskip
Here are quick proofs of Theorem \ref{T1} (1) $\Rightarrow$ (2) and (2) $\Leftrightarrow$ (3) for the sake of completeness. Remark that the former can also be derived as a consequence of a more general fact \cite[Corollary 3.4]{Shlyakhtenko:TAMS04}.

\medskip\noindent
{\bf Proof of Theorem 1 (1) $\Rightarrow$ (2):} 
Suppose, on the contrary, that there exists a sequence $t_n$ of real numbers so that $\sigma_{t_n}^{\varphi_c} \longrightarrow \Id_{M_c}$ in $\Aut(M_c)$ as $n\to\infty$ but $t_n$ does not converge to $0$ in the usual topology as $n\to\infty$. Let $\lambda^{\varphi_c} \colon \mathbb{R} \to \widetilde{M_c} = M_c\rtimes_{\sigma^{\varphi_c}}\mathbb{R}$ be the canonical unitary representation. Passing to a subsequence, we may assume that there is a positive constant $\varepsilon > 0$ so that $|t_n| \geq 3\varepsilon$ for all $n$. Then the regular representation $\lambda \colon \mathbb{R} \curvearrowright L^2(\mathbb{R})$ enjoys that $\Vert\lambda(t_n)\chi_{[-\varepsilon,\varepsilon]}-\zeta\chi_{[-\varepsilon,\varepsilon]}\Vert_2^2 \geq 2\varepsilon$ for all $n$ and all $\zeta \in \mathbb{C}$. It follows that the sequence $\lambda^{\varphi_c}(t_n)$ does never define a scalar in $(\widetilde{M_c})^\omega$. Set $E_{\varphi_c}
:= (\varphi_c\botimes \Id)\!\upharpoonright_{\widetilde{M_c}}$, a positive scalar multiple of faithfull normal conditional expectation onto $\mathbb{C}1_{M_c}\rtimes_{\sigma^{\varphi_c}}\mathbb{R}$. With a faithful normal state $\psi$ on $\mathbb{C}1_{M_c}\rtimes_{\sigma^{\varphi_c}}\mathbb{R}$ we have $\Vert \lambda^{\varphi_c}(t_n)x - x\lambda^{\varphi_c}(t_n)\Vert_{\psi\circ E_{\varphi_c}} = \Vert \sigma_{t_n}^{\varphi_c}(x) - x\Vert_{\varphi_c} \longrightarrow 0$ as $n\to\infty$ for every $x \in M_c$ so that the sequence $\lambda^{\varphi_c}(t_n)$ defines a non-scalar element of $(\widetilde{M_c})'\cap(\widetilde{M_c})^\omega$, a contradiction. \qed   

\medskip\noindent
{\bf Proof of Theorem 1 (2) $\Leftrightarrow$ (3):} 
This follows from the equivalence between $\sigma_{t_n}^\varphi \longrightarrow \Id_M$ in $\Aut(M)$ as $n\to\infty$ and $(\sigma_{t_n}^{\varphi_1},\sigma_{t_n}^{\varphi_2}) \longrightarrow (\Id_{M_1},\Id_{M_2})$ in $\Aut(M_1)\times\Aut(M_2)$ as $n \to\infty$. Firstly, the linear span of the identity $1$ and all alternating words in $\Ker(\varphi_k)$, $k=1,2$, forms a dense subspace of the standard form $L^2(M)$ which can be understood as the completion of $M$ with respect to the norm $\Vert\,-\,\Vert_\varphi$. Secondly, the free independence of $M_1,M_2$ with respect to $\varphi$ together with the formula $\sigma_t^\varphi = \sigma_t^{\varphi_1} \star \sigma_t^{\varphi_2}$ (see \cite{Barnett:PAMS95},\cite{Dykema:Crelle94}) enables us to see that $\Vert \sigma_{t_n}^\varphi(x) - x\Vert_\varphi \leq (\max_{1\leq i\leq l}\Vert x_i\Vert_\infty)^{l-1}\sum_{i=1}^l \Vert \sigma_{t_n}^{\varphi_{k_i}}(x_i)-x_i\Vert_{\varphi_{k_i}}$ for every alternating word $x = x_1\cdots x_l$ with $x_i \in \Ker(\varphi_{k_i})$. The equivalence is immediate from these facts (thanks to \cite[Theorem IX.1.15, Proposition IX.1.17]{Takesaki:Book2}). \qed 

\medskip
In closing of this section we give a simple remark explaining Corollary \ref{C2}. 

\begin{remark}\label{R9}
The corollary is indeed a particular case of Theorem \ref{T1}, since {\it any free Araki--Woods factor $\Gamma(\mathcal{H}_\mathbb{R},U_t)''$ with its distinguished state $\varphi_U$ can be written as a free product of two non-trivial von Neumann algebras} (see the proof of \cite[Theorem 2.7]{Vaes:Bourbaki05} for a related claim). This should be known by experts, but we do give an explanation about this for the reader's convenience.

Let $U_t = \exp(\sqrt{-1}tA)$ with $A = \int_{-\infty}^{+\infty} s\,E_A(ds)$ be the Stone representation on the complexification $\mathcal{H}_\mathbb{R}+\sqrt{-1}\mathcal{H}_\mathbb{R}$. The unitary conjugation $J \colon \xi + \sqrt{-1}\eta \mapsto \xi - \sqrt{-1}\eta$ for $\xi + \sqrt{-1}\eta \in \mathcal{H}_\mathbb{R}+\sqrt{-1}\mathcal{H}_\mathbb{R}$ enjoys the property that $JE_A(B)J = E_A(-B)$ for each Borel subset $B$ of $\mathbb{R}$. This shows that, if $\sharp(\Sp(A) \cap (0,+\infty)) \leq 1$, then $\Sp(A)$ must be either $\{0\}$, $\{-s,s\}$ or $\{-s,0,s\}$ with $s > 0$; hence the desired free product decomposition is obtained in the proof of \cite[Theorem 6.1]{Shlyakhtenko:PacificJMath97}. If $\sharp(\Sp(A) \cap (0,+\infty)) \geq 2$, then $\Sp(A) \cap (0,+\infty)$ is decomposed into two non-trivial Borel subsets $B_1, B_2$. Set $P_1:= E_A(-B_1 \cup\{0\}\cup B_1)$, $P_2 := E_A(-B_2\cup B_2)$, both of which commute with $U_t$ and $J$. Thus we have $(\mathcal{H}_\mathbb{R},U_t) = (P_1\mathcal{H}_\mathbb{R},U_t\!\upharpoonright_{P_1\mathcal{H}_\mathbb{R}})\oplus(P_2\mathcal{H}_\mathbb{R},U_t\!\upharpoonright_{P_2\mathcal{H}_\mathbb{R}})$ so that $\Gamma(\mathcal{H}_\mathbb{R},U_t)''$ becomes a free product of two free Araki--Woods factors (see \cite[Theorem 2.11]{Shlyakhtenko:PacificJMath97}).

In this way, almost all general results on free Araki--Woods factors follow, as particular cases, from the corresponding ones on free product von Neumann algebras; see \cite{Ueda:AdvMath11,Ueda:MRL11,Ueda:arXiv:1207.6838v3}. Only two non-trivial facts \cite[Theorem 5.4]{Shlyakhtenko:PacificJMath97},\cite[Theorem 4.8]{Shlyakhtenko:Crelle98}, both of which heavily depend upon `matricial models', have not been re-proved in the general framework of free product von Neumann algebras. These lacks seem to be related to the question \cite[\S5.4]{Ueda:AdvMath11}.
\end{remark} 

\section{Appendix: Bernoulli crossed products} 

Throughout this section, we follow the notation rule, etc.~in \cite[\S\S2.5]{VaesVerraedt:arXiv:1408.6414}, differently from the other sections. Let $\Lambda$ be a countable group acting on a countable set $I$ such that {\it $\Lambda \curvearrowright I$ has no invariant mean}, and $(P,\phi)$ be a non-trivial von Neumann algebra equipped with a faithful normal state $\phi$. Let $(P,\phi)^I\rtimes\Lambda$ (or $P^I\rtimes\Lambda$ for short) be the (generalized) Bernoulli crossed product, see e.g.~\cite[\S\S2.5]{VaesVerraedt:arXiv:1408.6414}. Set $\varphi := \phi^I\circ E_{P^I}$ with the canonical conditional expectation $E_{P^I} \colon P^I\rtimes\Lambda \to P^I$.  

\begin{lemma}\label{L10} For every countable subgroup $G$ of $\mathbb{R}$, any central sequence {\rm(}see \cite[Definition XIV.3.2]{Takesaki:Book3}{\rm)} in $(P^I\rtimes\Lambda)\rtimes_{\sigma^\varphi}G$  is equivalent to an {\rm(}operator norm-{\rm)}bounded one in $(\mathbb{C}1\rtimes\Lambda)\rtimes_{\sigma^\varphi}G = \mathbb{C}1\botimes L(\Lambda\times G)$. 
\end{lemma} 
\begin{proof} The idea used in the proof of \cite[Lemma 2.7]{VaesVerraedt:arXiv:1408.6414} works for proving this lemma. Let $x_n$ be a central sequence in $(P^I\rtimes\Lambda)\rtimes_{\sigma^\varphi}G$ . Consider those $x_n$ as vectors in the standard Hilbert space $L^2((P^I\rtimes\Lambda)\rtimes_{\sigma^\varphi}G) \cong [L^2((P,\phi)^I\ominus\,\mathbb{C})\botimes \ell^2(\Lambda)\botimes \ell^2(G)] \oplus [\ell^2(\Lambda)\botimes \ell^2(G)]$. Remark that this Hilbert space decomposition is given by the conditional expectation from $(P^I\rtimes\Lambda)\rtimes_{\sigma^\varphi}G$ onto $(\mathbb{C}1\rtimes\Lambda)\rtimes_{\sigma^\psi}G$ defined to be the restriction of $\phi^I\botimes \Id\botimes \Id$ to $(P^I\rtimes\Lambda)\rtimes_{\sigma^\varphi}G$ and also that the Bernoulli action commutes with the modular action associated with $\phi^I$. Hence, as in the proof of \cite[Lemma 2.7]{VaesVerraedt:arXiv:1408.6414} ({\it n.b.}~one of the keys there is that any tensor product representation of non-amenable one with arbitrary one must be non-amenable again; see e.g.~\cite[Proposition 2.7]{Stokke:MathScand06}), we see that the $x_n$ is equivalent to the $(\phi^I\botimes \Id\botimes \Id)(x_n)$ in $(\mathbb{C}1\rtimes\Lambda)\rtimes_{\sigma^\psi}G = \mathbb{C}1\botimes L(\Lambda\times G)$. Hence we are done.      
\end{proof} 

With the above lemma, one can prove the next proposition in the essentially same way as in the proof of Theorem \ref{T1} (1) $\Leftrightarrow$ (2).  

\begin{theorem}\label{T11} If $P^I\rtimes\Lambda$ is a full factor of type III$_1$, then the following conditions are equivalent{\rm:}
\begin{itemize} 
\item[(1)] The continuous core $(P^I\rtimes\Lambda)\rtimes_{\sigma^\varphi}\mathbb{R}$ of $P^I\rtimes\Lambda$ is a full factor. 
\item[(2)] The $\tau$-invariant $\tau(P^I\rtimes\Lambda)$, i.e., the weakest topology on $\mathbb{R}$ making the mapping $t \in \mathbb{R} \mapsto \sigma_t^\phi \in \Aut(P)$ continuous in this particular case, is the usual topology on $\mathbb{R}$. 
\end{itemize} 
\end{theorem}
\begin{proof}
It was shown in \cite[Lemma 2.7]{VaesVerraedt:arXiv:1408.6414} that $t_n \longrightarrow 0$ in $\tau(P^I\rtimes\Lambda)$ if and only if $\sigma^\varphi_{t_n} \longrightarrow \Id$ in $\Aut(P^I)$, which is easily seen to be equivalent to that $\sigma_{t_n}^\phi \longrightarrow \Id$ in $\Aut(P)$. Hence it suffices to prove (2) $\Rightarrow$ (1) as in Theorem \ref{T1}. In fact, the proof of Theorem \ref{T1} (1) $\Rightarrow$ (2) works by replacing $\varphi_c$ there with $\phi^I$. We will explain how to modify the proof of Theorem \ref{T1} (2) $\Rightarrow$ (1).  

We prove its contraposition. Namely, by Lemma \ref{L6} we assume that there exists a non-trivial strongly central sequence $x_n$ in $(P^I\rtimes\Lambda)\rtimes_{\sigma^\varphi_T}\mathbb{Z}$ for some $T > 0$. (See \cite[Definition XIV.3.2]{Takesaki:Book3} for the notion of strongly central sequences.) By Lemma \ref{L10} we may and do assume that all the $x_n$ fall in $(\mathbb{C}1\rtimes\Lambda)\rtimes_{\sigma^\varphi_T}\mathbb{Z} = \mathbf{C}1\botimes L(\Lambda\times\mathbb{Z})$. As in the proof of Theorem \ref{T1} (2) $\Rightarrow$ (1), it suffices to prove that there exists a sequence $k_l$ in $\mathbb{Z}\setminus\{0\}$ such that $\Vert x - \sigma_{k_l T}^\varphi(x)\Vert_\varphi \longrightarrow 0$ as $l\to\infty$ for every $x \in P^I\rtimes\Lambda$. Suppose, on the contrary, that this is not the case. Since $\mathbb{C}1\rtimes\Lambda \subseteq (P^I\rtimes\Lambda)_\varphi$, the fixed-point algebra of the modular action $\sigma^\varphi$, the same argument as in Claim in the proof of Theorem \ref{T1} (2) $\Rightarrow$ (1) actually works for proving that $\Vert x_n - (\Id\botimes \Id\botimes \tau_\mathbb{Z})(x_n)\Vert_2 \longrightarrow 0$ as $n\to\infty$. Set $y_n := (\Id\botimes \Id\botimes \tau_\mathbb{Z})(x_n) \in \mathbb{C}1\rtimes\Lambda \subseteq P^I\rtimes\Lambda$. Since we have assumed that $P^I\rtimes\Lambda$ is full, by \cite[Theorem XIV.3.8]{Takesaki:Book2} the $y_n$ (and hence the $x_n$) must be trivial, a contradiction.  
\end{proof}

So far, we have established, for every full type III$_1$ factor whose $\tau$-invariant is already computed, that the $\tau$-invariant is the usual topology if and only if the continuous core is full. Therefore, one may conjecture that this is true even for any full type III$_1$ factor. Actually, this question seems important from the theoretical point of view, and we are still working on this general question.

\end{document}